\documentclass[reqno,12pt]{amsart}
\usepackage[utf8]{inputenc} 
\allowdisplaybreaks
\usepackage[margin=1in]{geometry}
\usepackage{amsfonts}
\usepackage[all]{xy}
\usepackage{enumitem}
\usepackage{mathtools}
\usepackage{amssymb,amsmath,url}
\usepackage{setspace}
\usepackage{amsthm}
\usepackage{cite}
\usepackage{color}
\usepackage{mathrsfs}
\usepackage{bbm}
\usepackage{tikz}
\usepackage{tikz-cd}

\DeclareMathOperator{\vol}{Vol}

\numberwithin{equation}{section}
\usepackage[normalem]{ulem}
\usepackage{graphicx}
\usepackage{lmodern}
\graphicspath{ {images/} }
\usepackage{pdfpages}
\usepackage{MnSymbol}

\newcommand{\diag}{\mathrm{diag}}
\newcommand{\volfs}{\vol_{{\mathrm{FS}}}}

\newcommand{\abs}[1]{\left\vert#1\right\vert}
\newcommand{\norm}[1]{\left\Vert#1\right\Vert}

\newcommand{\bb}{\mathfrak{B}}
\newcommand{\cc}{\mathfrak{C}}
\newcommand{\p}{\mathbb{P}}
\newcommand{\n}{\mathsf{N}}
\newcommand{\s}{\mathsf{S}}
\newcommand{\T}{\mathsf{T}}
\newcommand{\D}{\mathsf{D}}

\newcommand{\ol}{\overline}

\newcommand{\rl}{\mathbb{R}}
\newcommand{\cx}{\mathbb{C}}

\newcommand{\ipr}[1]{\left\langle #1 \right\rangle}
\newcommand{\gipr}[1]{\left( #1 \right)}
\newcommand{\dist}{\mathrm{dist}}

\newcommand{\dbar}{\overline{\partial}}

		\newtheorem{thm}{Theorem}[section]

	\newtheorem{lem}[thm]{Lemma}

		\newtheorem{prop}[thm]{Proposition}
	
	\newtheorem*{thm*}{Theorem}
    	\newtheorem*{lem*}{Lemma}

\newtheorem*{propa}{Proposition A}
\newtheorem*{propb}{Proposition B}
\newtheorem*{propc}{Proposition C}
	\theoremstyle{definition}

\title[Volume of a Quadric]{Volume of the domain bounded by
a Hermitian quadric in complex projective space}
\author{Joyita Banerjee Ganguly}
\author{Debraj Chakrabarti}
\author{Meera Mainkar}
\address{Department of Mathematics, Central Michigan University, 
Mt. Pleasant, MI 48859,USA}
\email{baner1j@cmich.edu}
\email{chakr2d@cmich.edu}
\urladdr{https://people.se.cmich.edu/chakr2d/} 
\email{maink1m@cmich.edu}

 \thanks{Debraj Chakrabarti was partially supported by a grant from the NSF (DMS--2153907) and a grant from the Simons foundation (No. 706445).}       
	\begin{document}

\begin{abstract} 
We  compute explicitly the Riemannian 
volume, with respect to the Fubini-Study metric, of a domain bounded by a Hermitian quadric in complex projective space. The volume is a rational function of the eigenvalues 
of the defining quadratic form. 
\end{abstract}
	 	\maketitle
\section{Introduction}
\subsection{Volumes of Hermitian Quadrics} Let $A$ be an $(n+1)\times (n+1)$ Hermitian matrix, $n\geq 1$. We define the \emph{quadric domain}  determined by $A$ to be the open subset of the $n$-dimensional complex projective space $\cx\p^n$ given by
\begin{equation}
    \label{eq-omegaa}
    \Omega(A)=\{[z]\in \cx\p^n: \ipr{Az,z}>0\},
\end{equation}
where for a point $z\in \cx^{n+1}\setminus\{0\}$, we denote by $[z]\in \cx\p^n$ its projection  under the quotient map $\cx^{n+1}\setminus\{0\}\to \cx\p^n$, and $\ipr{z,w}=\sum_{j=1}^{n+1}z_j \ol{w_j}$ is the standard 
inner product of $\cx^{n+1}$. Quadric domains and their boundaries (the \emph{Hermitian hyperquadrics,} see \cite{cartanb}) are fundamental model domains in complex analysis and CR geometry (see, e.g.,\cite{ebenfelt,boggess}). For example, when the matrix $A$ has one negative eigenvalue and the other $n$ eigenvalues are positive, the domain $\Omega(A)$ is biholomorphic to the unit ball of $\cx^n$, undoubtedly  the single most important domain in  complex analysis (see \cite{rudin}).

Complex projective space admits the Fubini-Study metric, a Kähler metric invariant under the projective unitary self-maps of $\cx\p^n$. In  this paper we  compute the Riemannian volume of the quadric domain $\Omega(A)$ with respect to the Fubini-Study metric. The spectral theorem and the invariance of metric under the unitary group together imply that the volume is a  function of the eigenvalues of  $A$ (which are  real since $A$ is Hermitian), and this function
is symmetric separately in the  positive and the negative eigenvalues. Interestingly,  the  explicit formula for the volume (see \eqref{eq-explicit} below) turns out to be a \emph{rational} function of these eigenvalues. In contrast, in the real elliptic plane $\rl\p^2$  the analogous formula for the area of the region bounded by a conic is
transcendental in the eigenvalues of the defining matrix; see \cite{realprojective},
where the area is found to be a complete elliptic integral with the eigenvalues as parameters.
As far as we have been able to ascertain, the existence of a rational closed form expression for 
the volume of $\Omega(A)$ and its precise expression proved here has not appeared  in the literature before. 

\begin{thm*}\label{Th:Main theorem} Let $n$ be a positive integer and 
let $p$, $q$ and $r$ be nonnegative integers, such  $n=p+q+r-1$. Let $A$ be an $(n+1)\times(n+1)$ Hermitian matrix with $p$ positive eigenvalues $\mu_1,\ldots,\mu_p$, $q$ negative eigenvalues $-\nu_1,\ldots,-\nu_q$, (where each $\nu_j>0$) and $r$  eigenvalues each equal to 0. 
Then the Riemannian volume of  quadric domain $\Omega(A)\subset \mathbb{CP}^{n}$ with respect
to the Fubini-Study metric is given by
\begin{equation}
\label{eq-explicit}
\volfs(\Omega(A))=\volfs(\cx\p^n)\cdot\frac{\s_{(p,q)}(\mu,\nu)}{\D_{(p,q)}(\mu,\nu)},
\end{equation}
where  each of $\s_{(p,q)}$ and $\D_{(p,q)}$ is a homogeneous polynomial of degree $pq$ in the 
$p+q$ variables 
$\mu_1,\ldots,\mu_p, \nu_1,\ldots,\nu_q $ with nonnegative integer coefficients, which are given below in \eqref{eq-Sdef} and \eqref{eq-Ddef} respectively.
\end{thm*}

\subsection{The denominator $\D_{(p,q)}$}
Let $x=(x_1,\dots,x_p)$ and $y=(y_1,\dots, y_q)$ be $p$- and $q$-tuples of variables. The denominator $\D_{(p,q)}$ in \eqref{eq-explicit} is the polynomial in the variables $x_1,\dots,x_p, y_1,\dots, y_q$ given by
\begin{equation}
    \label{eq-Ddef}
    \D_{(p,q)}(x,y)={\prod\limits\limits_{k=1}^{q}\prod\limits\limits_{j=1}^{p}(x_j+y_k)},
\end{equation}
where as usual an empty product is interpreted as 1. Clearly,
$\D_{(p,q)}$ is a polynomial with nonnegative integer coefficients, and is symmetric in the $p$ variables $(x_1,\dots, x_p)$ as well as the $q$ variables $(y_1,\dots, y_q)$.

\subsection{Schur Polynomials} In order to describe the numerator
$\s_{(p,q)}$ we begin by recalling the classical \emph{Schur polynomials} (see, e.g., \cite{macdonald} for a detailed survey, \cite{fulton} for a masterly combinatorial approach,  and \cite{egge} for an  elementary account). These arise in mathematics in many places, for example in representation theory as characters of irreducible polynomial representations of the general linear groups. 

Recall 
that a \emph{partition} $\lambda=(\lambda_1,\dots, \lambda_n)$ is a tuple of nonnegative integers satisfying $\lambda_1\geq \lambda_2\geq \dots \geq \lambda_n$. The nonzero entries in the tuple $\lambda$ are referred to as the \emph{parts} of $\lambda$, and we identify two partitions if
they have the same parts. In particular, we admit a unique partition with zero parts, the empty partition, denoted by $0$.

Let $n \ge 1$ and $\lambda$ be a partition with at most $n$ parts. The \emph{Schur polynomial} $s_\lambda(x_1, \ldots, x_n)$ is defined as
\begin{align}
    s_\lambda(x_1, \ldots, x_n) = \frac{a_{\lambda+\delta_n}(x_1, \ldots, x_n)}{a_{\delta_n}(x_1, \ldots, x_n)},\label{eq-defschurpoly}
\end{align}
where $(x_1, \ldots, x_n)$ is a tuple of variables, $\delta_n$  is the partition $(n-1, n-2, \ldots, 0)$, 
\begin{align}
    a_{\delta_n}(x_1, \ldots, x_n) = \det(x_i^{\delta_n(j)})_{1 \le i,j \le n}= \det(x_i^{n-j})_{1 \le i,j \le n}=
{\displaystyle \prod _{1\leq i<j\leq n}(x_{i}-x_{j})},\label{eq-defadeltan}
\end{align}
is the \emph{Vandermonde alternating polynomial}, and 
\begin{align}
    a_{\lambda+\delta_n}(x_1,\cdots,x_n)=\det(x_\ell^{\lambda_j+\delta_n(j)})_{1\le j,\ell \le n}= \det(x_\ell^{\lambda_j+n-j})_{1\le j,\ell \le n}.\label{eq-defalambdaplusdeltan}
\end{align}
Since a determinant is an alternating function of its columns, and the partition $\lambda+\delta_n$ has distinct parts, it is easy to see that the polynomial $a_{\lambda+\delta_n}$ is alternating and nonzero.
Since $a_{\lambda+\delta_n}$ vanishes if two of the variables $x_i, x_j$ are equal, it therefore is divisible by $x_i-x_j$ for $i\not=j$, and therefore by their product the 
Vandermonde polynomial $a_{\delta_n}$. Therefore $s_\lambda$ is indeed a polynomial, and since both 
numerator and denominator of \eqref{eq-defschurpoly} are alternating, it follows that $s_\lambda$ is symmetric. 
For the empty partition, one sets $s_{0}\equiv 1$,  consistently with \eqref{eq-defschurpoly}.

\subsection{The numerator $\s_{(p,q)}$}
Denote by $\bb(p,q)$ the collection of partitions with at most $p$ parts, each of  the parts being at most $q$. The \emph{conjugate partition} $\lambda'$ to the partition $\lambda$ is defined by taking $\lambda_k'$ to be the largest integer $j$ such that $\lambda_j\geq k$, i.e.  $\lambda'_k = \#\{j: \lambda_j \ge k\}.$ This has a well-known interpretation  in terms of a transposed Ferrer's diagram. 
For $\lambda\in \bb(p,q) $ (see below) let $\lambda_{p,q}^*$ be the partition given by 
$(\lambda_{p,q}^*)_j= p-\lambda_{q-j+1}'$, i.e.,
\begin{equation}
\label{eq-starpqdef}
\lambda_{p,q}^*=(p-\lambda_q', \dots, p- \lambda_1'), \quad\text{ if $q\geq 1$} \quad \text{  and } \lambda_{p,0}^*=0.
\end{equation}
It is not difficult to see that $\lambda_{p,q}^*\in \bb(q,p).$ Introduce a subcollection of $\bb(p,q)$ by setting
\begin{equation}
    \label{eq-cdef}
    \cc(p,q)=\{\lambda=(\lambda_1,\dots, \lambda_p)\in \bb(p,q): \lambda_1=q\},
\end{equation}
consisting of those partitions in $\bb(p,q)$ in which the largest part is $q$. 
The numerator $\s_{(p,q)}$ in \eqref{eq-explicit} is given by
\begin{equation}
\label{eq-Sdef}
\s_{(p,q)}(x,y)= \sum_{\lambda\in \cc(p,q)}s_\lambda(x) s_{\lambda^*_{p,q}}(y),
\end{equation}
where $x$ and $y$ have the same meaning as in \eqref{eq-Ddef}. This completes the description of the numerator and denominator in \eqref{eq-explicit}.

\subsection{Related results and special cases in the literature} The only quadrics that bound regions of finite volume in Euclidean space are 
ellipsoids. The familiar formulas $\pi r^2$ for the area of a planar disc, $\frac{4}{3}\pi r^3$ for the volume of a three-dimensional ball and $\pi ab$ for the area in the plane bounded by
an ellipse
are  found in the writings of Euclid and Archimedes. Archimedes' formula for the area bounded by an ellipse and its proof (see \cite[pp. 113 ff.]{archimedes}) can be generalized to higher dimensions: for some positive integer $n$, if $Q$ is an $n\times n$ strictly positive definite symmetric matrix with eigenvalues $\lambda_1,\dots, \lambda_n$ (repeated according to multiplicity), then 
\begin{equation}
    \label{eq-ellipsoid}
    \vol_{\rl^n}\left\{x\in \rl^n: \gipr{Qx,x}< 1\right\} =\dfrac{\pi^{\frac{n}{2}}}{\Gamma\left(\frac{n}{2}+1\right)}\cdot  \frac{1}{\sqrt{\lambda_1 \cdots\lambda_n}}=\dfrac{\pi^{\frac{n}{2}}}{\Gamma\left(\frac{n}{2}+1\right)}\cdot  \frac{1}{\sqrt{\det Q}} ,
\end{equation}
where $\gipr{x,y}= \sum_{j=1}^n x_j y_j$ is the standard inner product of $\rl^n$,  and the second form of the formula
allows us to compute the  volume of an ellipsoid directly from the coefficients of its algebraic equation.

In {complex} Euclidean space, \eqref{eq-ellipsoid} is simplified: for an $n\times n$
strictly positive definite Hermitian matrix $A$ with eigenvalues $\lambda_1,\dots, \lambda_n$, we have, 
\[\vol_{\cx^n}\left\{z\in \cx^n: \ipr{Az,z}< 1\right\} =\dfrac{\pi^n}{n!}\cdot  \frac{1}{{\lambda_1 \cdots\lambda_n}}=\dfrac{\pi^n}{n!}\cdot  \frac{1}{{\det A}},\]
where $\ipr{z,w}=\sum_{j=1}^{n}z_j \ol{w_j}$ is the standard 
inner product of $\cx^{n}$. 

The simplicity of the complex formula above compared to the real formula \eqref{eq-ellipsoid} is a general phenomenon 
in Hermitian geometry, at the bottom of which lies the  relation between the real and the complex determinant of a $\cx$-linear map: $\det_\rl T =\abs{\det_\cx T}^2.$ Other examples of such simplification include the ``complex Pythagorean theorem", and  the rationality  of \eqref{eq-explicit}, in contrast to the transcendental nature of the analogous formula in real elliptic space. 

  The Fubini-Study metric makes $\cx\p^n$
into a homogeneous Kähler manifold, and the resulting ``Hermitian elliptic space" is of fundamental importance in algebraic and differential geometry  
(see, e.g., \cite{gh,mum}). The importance of metric quantities  in Hermitian elliptic geometry is underscored by the  well-known Wirtinger formula: for an $r$-dimensional smooth projective variety $Z$ in $\cx\p^n$, we have
\[\volfs{Z}=\volfs{\cx\p^r}\cdot\deg(Z) , \]
where $\volfs$  denotes the $2r$-dimensional Riemannian volume with respect to the Fubini-Study metric,
and $\deg(Z)$ is the degree of $Z$. It follows that the volume of a smooth projective variety is  computable algebraically
from the  equations of $Z$. In view of this, it is surely interesting to try to compute 
volumes  of  objects in projective space, and in particular of open sets such as \eqref{eq-omegaa} with simple algebraic definitions. Notice that \eqref{eq-explicit} also allows a purely algebraic
computation of the volume, provided we can determine the eigenvalues. 

A few special cases of volumes of quadric domains in $\cx\p^n$ are known in the literature, though not in the form \eqref{eq-explicit}. 
In $\cx\p^1$, the only 
interesting quadric domain is a disc, corresponding to  $n=p=q=1$. The Fubini-Study metric can be 
identified with the induced round metric of a sphere in $\rl^3$, traditionally normalized to radius $\frac{1}{2}$,  and the  formula $\pi \sin^2 \rho$ for the area of a metric disk of geodesic radius $\rho$ on 
the sphere is elementary and well-known.
Generalizing the disc, one can consider domains of the form
\begin{equation}
    \{w\in \cx\p^n: \dist_{FS}(Z,w)<\rho\}, \label{eq-tube}
\end{equation} 
where $\dist_{FS}$ is the Fubini-Study distance, $Z$ is a projective-linear subspace
of $\cx\p^n$ and $\rho$ is a positive number. Using the explicit formula (see below, \eqref{eq-fubinistudy}) for the Fubini-Study
distance, one sees  that this set is a quadric domain $\Omega(A)$ (cf. \cite[\S 179]{cartanb}), with the special feature that the positive eigenvalues of $A$ are all equal to each other, and so are the negative eigenvalues. 
The set \eqref{eq-tube} is also a \emph{tube} in the 
sense of Weyl in $\cx\p^n$ and its volume can be computed as a function of $\rho$ by
a version of Weyl's tube formula in $\cx\p^n$ (see \cite[Theorem~7.20 and formula (7.58)]{gray}). This gives the volume as an integral over the projective-linear subspace $Z$. In the case when $Z$ is a projective hyperplane in $\cx\p^n$, further simplifications occur (see \cite[Theorem~7.22 and formula (7.62)]{gray}), and the integral disappears. However, we would like to emphasize here 
that 
quadric domains of the form \eqref{eq-tube} are  very special and most of them  cannot be  represented as \eqref{eq-tube}, and
further, the expressions for the volume
obtained from the complex projective tube formula
do not reveal the algebraically simple character of the representation of the volume in terms of the eigenvalues.

\subsection{Plan of this paper} The proof of the main result given here is elementary, and is essentially a brute-force computation of the integral that gives the volume of $\Omega(A)$. The computation is simplified by finding a recursion formula for the volume in terms of $p$ and $q$ (the number of positive and negative eigenvalues) and then showing that the right hand side of \eqref{eq-explicit} is the only solution of the recursion relation. We begin by looking at some useful alternative ways of writing our main formula \eqref{eq-explicit} in Section~\ref{sec-alternative} below. Then  we give a bird's eye view of the main arguments of the proof in Section~\ref{sec-outline}, postponing the technicalities to the last section (Section~\ref{sec-computations}). 

\subsection{Questions and future directions} As mentioned above, the proof of \eqref{eq-explicit} offered here is computational, and  it would  be very interesting to find a more conceptual proof.
Such a proof will make more use of the symmetry of the metric of $\cx\p^n$ than we are doing here, and will make precise why the characters of the product 
group $U(p)\times U(q)$ occur in the numerator \eqref{eq-Sdef} and the denominator \eqref{eq-Ddef} of \eqref{eq-explicit} in the way they do.
 Notice here that $U(p)\times U(q)$ acts as a group of
holomorphic congruences of $\Omega(A)$, i.e. isometries of the ambient space $\cx\p^n$ preserving the set $\Omega(A)$, so it is not surprising that 
the representation theory of this group plays a role in the formula \eqref{eq-explicit}.

There are also at least two directions
in which one can seek to generalize the formula. First, one can ask for analogous formulas in other homogeneous complex manifolds with a notion of 
a quadric, such as in the projective realization of complex hyperbolic space. One can also ask if the volumes of  certain tube domains around algebraic varieties in $\cx\p^n$ admit interesting formulas in the spirit of \eqref{eq-explicit}. We plan to pursue these questions and generalizations in
future publications. 
\subsection{Acknowledgements} The second-named author became aware of the existence of the interesting formula 
\eqref{eq-explicit} during collaboration with Andy Raich and Phil Harrington on a problem 
in several complex variables (see \cite{andypaper}), when he wanted to know the probability that the Levi form of a hypersurface in a random direction
at a non-pseudoconvex point is positive. He thanks Andy and Phil for their generosity and friendship throughout the years. He also thanks Ben Salisbury and Alperen Ergur for helpful discussions. 

\section{Alternative representations of $\s_{(p,q)}$ and $\D_{(p,q)}$}\label{sec-alternative}
\subsection{$\D_{(p,q)}$ in terms of Schur polynomials}It is a classical fact (see \cite[page 67]{macdonald}) that 
\begin{equation}
    \label{eq-dualcauchy}
     \D_{(p,q)}(x,y)= \sum_{\lambda\in \bb(p,q)}s_\lambda(x) s_{\lambda^*_{p,q}}(y).
\end{equation}
Since the Schur polynomials  have nonnegative 
integer coefficients (see \cite[Chapter 4]{egge}), thanks to the formula \eqref{eq-Sdef} for the numerator, it follows that the coefficient of a monomial in $\s_{(p,q)}$ is less than or equal to the coefficient of the same monomial in $\D_{(p,q)}$.
\subsection{$\s_{(p,q)}$ without Schur Polynomials} For a $p$-tuple $x=(x_1,\dots, x_p)$ and for any $1\leq \ell \leq p$, we define the $(p-1)$-tuple $x[\ell]$ as the one obtained 
from $x$ by removing its $\ell$-th entry:
    \begin{equation}\label{eq-y[k]}
        x[\ell]=\left(x_1,\cdots,x_{\ell-1},x_{\ell+1},\cdots,x_p\right). 
    \end{equation}
    In particular, $x[1]=(x_2,x_3,\cdots,x_p)$ and $x[p]=(x_1,x_2,\cdots,x_{p-1}).$

    \begin{prop}\label{prop-spqalternative}
    For non-negative integers $p,q$ we have
\begin{equation}
\s_{(p,q)}(x,y)
=\sum_{\ell=1}^p 
\frac{x_\ell^{p+q-1}}{\displaystyle\prod_{\substack{k=1\\k\ne\ell}}^p (x_\ell - x_k)}\cdot \D_{(p-1,q)}(x[\ell],y).\label{eq-Spq in terms of sum}
\end{equation}
\end{prop}
We will need the following lemma:
\begin{lem}
    \label{lem-combinatorial}
    The map $\xi\mapsto \xi[1]$, $(\xi_1,\cdots,\xi_p)\mapsto (\xi_2,\dots, \xi_p)$
    \begin{enumerate}
        \item is a bijection from $\mathfrak{C}(p,q)$ onto  $\mathfrak{B}(p-1,q)$, and
        \item  For a partition $\xi\in \mathfrak{C}(p,q)$, we have
         $(\xi[1])^\ast_{p-1,q}=\xi^\ast_{p,q}$, with notation as in \eqref{eq-starpqdef}. 
    \end{enumerate}

    \end{lem}
    \begin{proof}
\begin{enumerate}[wide]
\item 
    Let $\xi= (\xi_1, \xi_2, \dots, \xi_p)\in \mathfrak{C}(p,q)$, so that $\xi$ has at most $p$ parts and $\xi_1=q$ and $\xi_k\leq q$ for each $k$.
Then $\xi[1] = (\xi_2, \xi_3, \dots, \xi_p)$ has at most $p-1$ parts and each part is less than or equal to $q$, so $\xi[1]\in\mathfrak{B}(p-1, q)$. 

To show injectivity,  let $\xi,\eta\in \mathfrak{C}(p,q)$ be such that $\xi[1]=\eta[1]$,
       i.e. $(\xi_2,\cdots,\xi_p)=(\eta_2,\cdots,\eta_p)$.
        By definition of $\mathfrak{C}(p,q)$ we have $\xi_1=q=\eta_1,$ so we have $\xi=\eta,$
        completing the proof of injectivity.

To show surjectivity,  let $\eta\in \mathfrak{B}(p-1,q)$,
        and let $\xi=(q,\eta_1,\eta_2,\cdots,\eta_{p-1})$.
        Since $\eta$ has non-increasing parts with $\eta_i\le q,$ $\xi$ also has non-increasing parts and is a valid partition with at most $p$ parts and largest part equal to  $q$, so $\xi \in \mathfrak{C}(p,q).$ 
Since         \[ \xi[1]=(q,\eta_1,\eta_2,\cdots,\eta_{p-1})[1]=(\eta_1,\eta_2,\cdots,\eta_{p-1})=\eta,\]
        the mapping $\xi \mapsto \xi[1]$ is surjective.
\item 
        Let $\xi=(q, \xi_2,\cdots,\xi_p)\in \mathfrak{C}(p,q)$.
       We first note that $(\xi[1])'_j=\xi'_j-1$, for $1\le j\le q$, since
       \begin{align*}
           (\xi[1])'_j &=(\xi_2,\cdots,\xi_p)'_j=\#\ \{k:\xi_k\ge j, 2\le k\le p\}\\
           &=\#\ \{k:\xi_k\ge j, 1\le k\le p\}-1 \\& \text{ (since  $ \{k:\xi_k\ge j, 1\le k\le p\} =\{1\}\sqcup\{k:\xi_k\ge j, 2\le k\le p\}$ )}\\
           &=\xi_j'-1.
       \end{align*}
       Consequently \begin{align}
            (\xi[1])^\ast_{p-1,q}
              &=((p-1)-(\xi[1])'_q,\cdots,(p-1)-(\xi[1])'_1)\nonumber\\
              &=((p-1)-(\xi_q'-1),\cdots,(p-1)-(\xi_1'-1))\nonumber\\
              &=(p-\xi_q',\cdots,p-\xi_1')\nonumber\\
              &=\xi^\ast_{p,q}.\label{eq-conjugation relation}
       \end{align}
               \end{enumerate}

    \end{proof}

\begin{proof}[Proof of Proposition~\ref{prop-spqalternative} ]

Let $\xi \in \mathfrak{C}(p,q)$, where as in \eqref{eq-cdef},  $\mathfrak{C}(p,q) = \{\lambda \in \mathfrak{B}(p,q) : \lambda_1 = q\}$, so that 
$\xi_1 = q$ and $\xi$ has at most $p$ parts. From the determinant definition of $a_{\xi+\delta_p}$ in \eqref{eq-defalambdaplusdeltan}, we get by
expanding the determinant  by the first row:
\begin{align}
            a_{\xi+\delta_p}(x)=&\det \begin{pmatrix}
            x_1^{\xi_1+p-1} & x_2^{\xi_1+p-1} & \cdots & x_p^{\xi_1+p-1}\\
            x_1^{\xi_2+p-2} & x_2^{\xi_2+p-2} & \cdots & x_p^{\xi_2+p-2}\\
            \vdots          & \vdots          & \ddots & \vdots\\
            x_1^{\xi_p}     & x_2^{\xi_p}     & \cdots & x_p^{\xi_p}
        \end{pmatrix}=\sum\limits_{\ell=1}^p (-1)^{\ell-1} x_\ell^{p+\xi_1-1} a_{\xi[1]+\delta_{p-1}}(x[\ell])\nonumber\\ 
        &=\sum\limits_{\ell=1}^p (-1)^{\ell-1} x_\ell^{p+q-1} a_{\xi[1]+\delta_{p-1}}(x[\ell])\quad \text{(Since $\xi_1=q$)}.\label{eq-expansion of a xi plus deltan}
 \end{align}
Therefore
        \begin{align}
            \s_{(p,q)}(x,y)&= \sum_{\lambda\in \cc(p,q)}s_\lambda(x) s_{\lambda^*_{p,q}}(y) \quad \text{ using definition \eqref{eq-Sdef}}\nonumber \\
            &=\sum_{\xi\in \mathfrak{C}(p,q)}\frac{a_{\xi+\delta_p}(x)}{a_{\delta_p}(x)} \cdot s_{\xi_{p,q}^\ast}(y)\quad \text{ using definition of Schur polynomial \eqref{eq-defschurpoly} }\nonumber\\
            &=\frac{1}{a_{\delta_p}(x)}\sum_{\xi\in \mathfrak{C}(p,q)}a_{\xi+\delta_p}(x)\cdot s_{\xi_{p,q}^\ast}(y)\nonumber\\
           &= \frac{1}{a_{\delta_p}(x)}\sum_{\xi\in \mathfrak{C}(p,q)}\left[\sum\limits_{\ell=1}^p (-1)^{\ell-1} x_\ell^{p+q-1} a_{\xi[1]+\delta_{p-1}}(x[\ell])\right]\cdot s_{\xi_{p,q}^\ast}(y)\quad\text{ using \eqref{eq-expansion of a xi plus deltan})}\nonumber\\
           &= \frac{1}{a_{\delta_p}(x)} \sum_{\ell=1}^p (-1)^{\ell-1} x_\ell^{p+q-1} 
   \sum_{\xi \in \mathfrak{C}(p,q)} a_{\xi[1]+\delta_{p-1}}(x[\ell])\, s_{\xi_{p,q}^\ast}(y) \nonumber\\
&= \frac{1}{a_{\delta_p}(x)} \sum_{\ell=1}^p (-1)^{\ell-1} x_\ell^{p+q-1} 
   \sum_{\xi \in \mathfrak{C}(p,q)} s_{\xi[1]}(x[\ell])\, a_{\delta_{p-1}}(x[\ell])\, s_{\xi_{p,q}^\ast}(y),\nonumber \\
   & \text{by  definition of Schur polynomials \eqref{eq-defschurpoly}, $\quad a_{\xi[1]+\delta_{p-1}}(x[\ell]) = s_{\xi[1]}(x[\ell])\, a_{\delta_{p-1}}(x[\ell])$}.\nonumber\\
 &=  \frac{1}{a_{\delta_p}(x)}\sum\limits_{\ell=1}^p (-1)^{\ell-1} x_\ell^{p+q-1}a_{\delta_{p-1}}(x[\ell])\left[\sum_{\xi\in \mathfrak{C}(p,q)}s_{\xi[1]}(x[\ell])s_{\xi_{p,q}^\ast}(y)\right]. \label{eq-S(x,y)inter}
        \end{align}
Notice that 
\begin{align*}
    \sum_{\xi\in \mathfrak{C}(p,q)}s_{\xi[1]}(x[\ell])s_{\xi_{p,q}^\ast}(y)&=  \sum_{\xi\in \mathfrak{C}(p,q)}s_{\xi[1]}(x[\ell])s_{(\xi[1])_{p-1,q}^\ast}(y)\quad \text{using Part 2 of Lemma~\ref{lem-combinatorial}}\\
    &= \sum_{\xi[1]\in\mathfrak{B}(p-1,q)}s_{\xi[1]}(x[\ell])s_{(\xi[1])_{p-1,q}^\ast}(y) \quad \text{using Part 1 of Lemma~\ref{lem-combinatorial}}\\
    &=  \sum_{\lambda\in\mathfrak{B}(p-1,q)}s_{\lambda}(x[\ell])s_{\lambda_{p-1,q}^\ast}(y) =\D_{(p-1,q)}(x[\ell],y), \quad \text{ by \eqref{eq-dualcauchy}}
\end{align*}   and so    \begin{align}
    \eqref{eq-S(x,y)inter}
    &=\sum\limits_{\ell=1}^p (-1)^{\ell-1}x_\ell^{p+q-1}\frac{a_{\delta_{p-1}}(x[\ell])}{a_{\delta_p}(x)}\D_{(p-1,q)}(x[\ell],y)\nonumber\\
    &=\sum\limits_{\ell=1}^p (-1)^{\ell-1}x_\ell^{p+q-1}\frac{a_{\delta_{p-1}}(x[\ell])}{\prod\limits_{1\le i<j\le p}(x_i-x_j)}\D_{(p-1,q)}(x[\ell],y)\nonumber\\
    &=\sum\limits_{\ell=1}^p (-1)^{\ell-1}x_\ell^{p+q-1}\frac{a_{\delta_{p-1}}(x[\ell])}{\prod\limits_{\substack{1\le i<j\le p\\i\ne \ell\\j\ne\ell}}(x_i-x_j)\cdot\prod\limits_{i<\ell}(x_i-x_\ell)\cdot\prod\limits_{j>\ell}(x_\ell-x_j)}\D_{(p-1,q)}(x[\ell],y)\nonumber\\
    &=\sum\limits_{\ell=1}^p (-1)^{\ell-1}x_\ell^{p+q-1}\frac{a_{\delta_{p-1}}(x[\ell])}{a_{\delta_{p-1}}(x[\ell])\cdot (-1)^{\ell-1}\prod\limits_{\substack{k=1\\k\ne\ell}}^p(x_\ell-x_k)}\D_{(p-1,q)}(x[\ell],y)\nonumber\\
    &=\sum\limits_{\ell=1}^p (-1)^{\ell-1}x_\ell^{p+q-1}\frac{1}{(-1)^{\ell-1}\prod\limits_{\substack{k=1\\k\ne\ell}}^p(x_\ell-x_k)}\D_{(p-1,q)}(x[\ell],y)\nonumber\\
    &=\sum\limits_{\ell=1}^p \frac{x_\ell^{p+q-1}}{\prod\limits_{\substack{k=1\\k\ne\ell}}^p(x_\ell-x_k)}\D_{(p-1,q)}(x[\ell],y),\nonumber
\end{align}
completing the proof. 
\end{proof}

\section{Outline of the proof}\label{sec-outline}
\subsection{The metric and its volume form} Let us begin by recalling some formulas related to the Hermitian geometry of $\cx\p^n$ (for a detailed account, 
see \cite[pp. 159-161 and pp. 273-278]{kn2}).
For definiteness, in the rest of the paper, we will use a standard normalization of the Fubini-Study metric in which the holomorphic sectional 
curvature is 4. 
Then Fubini-Study distance between two points in $\cx\p^n$ is given by the formula
\begin{equation}
    \label{eq-fubinistudy}
    \dist_{FS}([z], [w])=\arccos \left(\frac{\abs{\ipr{z,w}}}{\norm{z}\cdot \norm{w}} \right), \quad z,w\in \cx^{n+1}\setminus \{0\},
\end{equation}
where $[z]$ denotes the one dimensional complex subspace of $\cx^{n+1}$ containing $z.$
The Fubini-Study distance 
is clearly invariant under the group of projectivized unitary automorphisms $PU(n+1)$ of $\cx\p^n$, i.e., maps of the form $[U]:\cx\p^n\to \cx\p^n$ given 
by $[U][z]=[Uz]$ where $U\in U(n+1)$, $z\in \cx^{n+1}\setminus \{0\}$.
It is well-known that \eqref{eq-fubinistudy} extends to a Kähler metric on $\cx\p^n$. A Kähler potential on the affine piece $\{z_{n+1}\not=0\}$ is given in terms of the inhomogeneous coordinates $\zeta_j =\dfrac{z_j}{z_{n+1}}, 1\leq j \leq n$ by $\varphi =\log(1+\abs{\zeta}^2)$, where  $\lvert\zeta\rvert^2=\lvert\zeta_1\rvert^2+\cdots+\lvert\zeta_n\rvert^2$. From this, one deduces
immediately formulas for the Kähler form $\omega =i\partial \dbar \varphi,$ as well as the volume form $\dfrac{\omega^{\wedge n}}{n!}$, which a 
computation shows is equal to  ${(1+\lvert\zeta\rvert^2)^{-(n+1)}}{dV(\zeta)}$, where $dV(\zeta)$ denotes the volume form of complex Euclidean space. Since
the hypersurface $\{z_{n+1}=0\}$ in $\cx\p^n$ has measure zero, it follows that 
for a measurable set $U\subset \cx\p^n$, we have the formula 
 \begin{equation}
        \volfs(U)=\int\limits\limits_{\sigma^{-1}(U)}\frac{dV(\zeta)}{(1+\lvert\zeta\rvert^2)^{n+1}}\label{eq-volfsdef}
        \end{equation}
        for the Fubini-Study volume,
        where $\sigma:\cx^n\to \mathbb {CP}^n$ is the parametrization of the affine piece $\{z_{n+1}\not=0\}$ given by $\sigma(\zeta)=[(\zeta_1,\dots,\zeta_n,1)]$. The Riemannian metric and its volume form are also invariant under the group $PU(n+1)$. 

        Our goal therefore is to compute the integral \eqref{eq-volfsdef} when $U=\Omega(A)$ in terms of the eigenvalues of $A$.
        Notice that when $A$ is positive definite (i.e. all eigenvalues of $A$ are strictly positive) then $\Omega(A)=\cx\p^n$. 
    In this case it is well known that 
    \begin{equation}\label{eq-cpnvol}
    \volfs(\mathbb{CP}^n)=\frac{\pi^n}{n!}.
    \end{equation}
   For $n=0$,  when $\mathbb{CP}^0$ is a singleton,   we let $  \volfs(\mathbb{CP}^n)=1$, so that \eqref{eq-cpnvol} still holds. 

\subsection{Reduction to nondegenerate diagonal form } Let $A$ be as in the main theorem of this paper, i.e., $A$ is a Hermitian square matrix of size 
$n+1$ where $n=p+q+r-1$, with $p=$ number of positive eigenvalues, $q=$ number of negative eigenvalues and $r=$ number of zero eigenvalues. Recall that 
the positive eigenvalues are denoted by $\mu_j, 1\leq j\leq p$ and the negative eigenvalues are denoted by $\nu_k, 1\leq k \leq q$. 
Thanks to the spectral theorem, there is a unitary matrix $U$ such that $U^{-1}AU=\Delta,$ where 
\begin{equation}
    \label{eq-Delta}\Delta=\diag(\mu, -\nu,0)=
\diag(\mu_1,\dots, \mu_p, -\nu_1,\dots, -\nu_q, \underbrace{0,\dots,0}_r).
\end{equation}
Then we have $\Omega(A)=[U](\Omega(\Delta))$ where
$[U]:\cx\p^n\to \cx\p^n$ is the projectivized unitary map $[z]\mapsto [Uz]$ induced by the unitary automorphism $U$. Therefore
\begin{equation}
    \label{eq-diagonal}
    \volfs(\Omega(A))=\volfs([U](\Omega(\Delta))=\volfs(\Omega(\Delta)),
\end{equation}
since the Fubini-Study metric (and therefore its Riemannian volume form) is invariant under projectivized unitary maps. To compute
$\volfs(\Omega(\Delta))$, let $D$ be the diagonal matrix of size $p+q$ given by
\begin{equation}
    \label{eq-matrixD}
    D=\diag(\mu,-\nu)= (\mu_1,\dots, \mu_p, -\nu_1,\dots, -\nu_q).
\end{equation}
We will prove the following:
\begin{prop} \label{prop-reduction} We have
\begin{equation}
    \label{eq-DeltatoD}
    \volfs(\Omega(\Delta))=\frac{\pi^r}{n(n-1)\dots (n-r+1)}\volfs(\Omega(D)),
\end{equation}
    where on the left we have $2n=2(p+q+r-1)$ real-dimensional volume and on the right $2(p+q-1)$ real-dimensional volume.
\end{prop}
We need an  integral representation of the volume.  Denote in the sequel by $\rl_+$ the set of positive real numbers, so that $\rl_+^n$ denotes the orthant in $\rl^n$ consisting of points with all coordinates strictly positive.   For $\gamma=(\gamma_1,\cdots,\gamma_{P})\in \mathbb{R}_+^{P}$ and $\delta=(\delta_1,\cdots,\delta_Q)\in \mathbb{R}_+^Q$ we define the set $E_{P,Q}(\gamma, \delta)\subset \mathbb{R}_+^{P+Q}$ as:
\begin{equation}
    E_{P,Q}(\gamma, \delta)=\left\{ (t_1, \dots, t_{P+Q}) \in \mathbb R^{P+Q}_+ : \sum_{j=1}^{P} \gamma_j t_j - \sum_{k=1}^{Q} \delta_k t_{P+k} > -1 \right\}.\label{eq-defsetE}
\end{equation}
\begin{lem}\label{lemma-integral}
  Let $p\ge 1, q,r\geq 0$, $x\in \rl_+^p, y\in \rl_+^q$,  and let 
  \[ \Delta=\diag(x_1,\cdots,x_p,-y_1,\cdots,-y_q, \underbrace{0, \dots, 0}_r) \]
be an $(n+1)\times (n+1)$ diagonal matrix, where let $n=p+q+r-1$. Let \begin{align}
 a &= \frac{1}{x_1}\cdot x[1]=\left(\frac{x_2}{x_1},\cdots,\frac{x_p}{x_1}\right)\in\mathbb{R}^{p-1}_+, \label{eq-deftuple"a"} \\
 \intertext{where $x[1]=(x_2,\cdots,x_p)\in \mathbb{R}_+^{p-1}$ is as in \eqref{eq-y[k]} and}\nonumber\\
 b &=\frac{1}{x_1}\cdot y =\left(\frac{y_1}{x_1},\cdots,\frac{y_q}{x_1}\right)\in\mathbb{R}^{q}_+\label{eq-deftuple"b"}.
\end{align}
Denoting $t'=(t_1,\dots, t_{p+q-1}),t''=(t_{p+q},\dots, t_n)$ we have
\begin{align}
       \volfs(\Omega(\Delta))=\pi^{n} \int\limits_{ E_{p-1,q}(a,b)}\int\limits_{\rl^r_+}  \frac{dV(t'')dV(t')}{\left(1 + \sum\limits_{j=1}^{n} t_j \right)^{n+1}},\label{eq-V(p,q)asintegral}
   \end{align} 
   where $E_{p-1,q}(a,b)$ as in \eqref{eq-defsetE}. 
\end{lem}
  \begin{proof}
  On the affine chart $U_1=\{[z]\in \cx\p^n:z_1\not=0\}$, introduce inhomogeneous coordinates $\xi_j=\dfrac{z_{j+1}}{z_1}, 1\leq j \leq n$. In these coordinates
  \begin{equation}
      \label{eq-tauinv}\Omega(\Delta)\cap U_1= \left\{\xi\in \cx^{n}: \sum_{j=1}^{p-1} a_j |\xi_j|^2 - \sum_{k=1}^{q} b_k |\xi_{p+k-1}|^2 > -1 \right\}.
  \end{equation}
Map $\cx^n$ to $\rl^+_n\times [0,2\pi)$ by setting $\xi_j=\sqrt{t_j}e^{i\theta_j}$. Notice that then the volume elements are related by 
\begin{equation}\label{eq-jacobian}
dV(\xi)=\dfrac{1}{2^n}dV(t)dV(\theta),
\end{equation}
where as usual $dV$ denotes the volume element of the Lebesgue measure (of an appropriate dimension). In the coordinates $(t_1,\dots, t_n, \theta_1,\dots, \theta_n)$, we have
\begin{align}
   \Omega(\Delta)\cap U_1&= \left\{\sum_{j=1}^{p-1} a_j t_j - \sum_{k=1}^{q} b_k t_{p+k-1} > -1, 0\leq \theta_j <2\pi \text{ for } 1\leq j \leq n\right\}\label{eq-tauinv2}\\
   &= \{(t',t'',\theta)\in \rl_+^{p+q-1}\times \rl_+^r\times [0,2\pi): t'\in E_{p-1,q}(a,b)\}\nonumber\\
   &=  E_{p-1,q}(a,b)\times \rl_+^r \times [0,2\pi)^n.\label{eq-tauinv3}
\end{align}
  Therefore, from  \eqref{eq-volfsdef}, the Fubini-Study volume of $\Omega(\Delta)$ is given by

\begin{align*}
    \volfs(\Omega(\Delta)) &= \int\limits_{\eqref{eq-tauinv}}\frac{dV(\xi)}{\left(1+ \sum_{j=1}^n\abs{\xi_j}^2\right)^{n+1}},\nonumber\\
    &= \frac{1}{2^n} \int_{\eqref{eq-tauinv2}}\frac{dV(t)dV(\theta)}{\left(1+ \sum_{j=1}^n t_j\right)^{n+1}} \quad \text{ using \eqref{eq-tauinv2} and \eqref{eq-jacobian}} \nonumber\\
    &= \frac{1}{2^n}\int_{[0,2\pi)^{n}}dV(\theta) \int\limits_{E_{p-1,q}(a,b)}\int\limits_{\rl^r_+}  \frac{dV(t'')dV(t')}{\left(1 + \sum\limits_{j=1}^{n} t_j \right)^{n+1}}\quad\text{ using \eqref{eq-tauinv3}},\\
    &= \pi^n \int\limits_{E_{p-1,q}(a,b)}\int\limits_{\rl^r_+}  \frac{dV(t'')dV(t')}{\left(1 + \sum\limits_{j=1}^{n} t_j \right)^{n+1}}, 
\end{align*}
completing the proof. 
  
\end{proof}
\begin{proof}[Proof of Proposition~\ref{prop-reduction}]
    Recalling that $t''=(t_{p+q},\dots, t_n)$, we have 
    \begin{align*}
       \int\limits_{\rl^r_+}\frac{dV(t'')}{\left(1 + \sum\limits_{j=1}^{n} t_j \right)^{n+1}}&=  \int\limits_{\rl_+^{r-1}}\left(\int_{t_n=0}^\infty{\left(1 + \sum\limits_{j=1}^{n-1} t_j +t_n\right)^{-(n+1)}}dt_n\right)dt_{p+q}dt_{p+q+1}\dots dt_{n-1}\\
       &=  \int\limits_{\rl_+^{r-1}}\left.\left(\frac{-1}{n}\cdot{\left(1 + \sum\limits_{j=1}^{n-1} t_j +t_n\right)^{-n}}\right\vert_{t_n=0}^\infty \right)dt_{p+q}dt_{p+q+1}\dots dt_{n-1}\\
       &= \frac{1}{n}\int\limits_{\rl_+^{r-1}} \left(1 + \sum\limits_{j=1}^{n-1} t_j\right)^{-n} dt_{p+q}dt_{p+q+1}\dots dt_{n-1}.
        \end{align*}
        Iterating this $r$ times we see that
        \begin{equation}
            \label{eq-r-iterate}
            \int\limits_{\rl^r_+}\frac{dV(t'')}{\left(1 + \sum\limits_{j=1}^{n} t_j \right)^{n+1}}= \frac{1}{n(n-1)\cdot \cdots \cdot(n-r+1)} \left(1 + \sum\limits_{j=1}^{p+q-1} t_j\right)^{-(p+q)}. 
        \end{equation}
        Therefore, by  Lemma~\ref{lemma-integral},
        \begin{align*}
            \volfs(\Omega(\Delta))&= \pi^r\cdot \pi^{n-r}  \int\limits_{E_{p-1,q}(a,b)}\left(\int\limits_{\rl^r_+}  {\left(1 + \sum\limits_{j=1}^{n} t_j \right)^{-(n+1)}}dV(t'')\right)dV(t')\\
            &= \frac{\pi^r}{n(n-1)\cdot \cdots \cdot(n-r+1)}\cdot \pi^{p+q-1}\int\limits_{E_{p-1,q}(a,b)} \left(1 + \sum\limits_{j=1}^{p+q-1} t_j\right)^{-(p+q)}dV(t')\\
            &=\frac{\pi^r}{n(n-1)\cdot \cdots \cdot(n-r+1)}\cdot \volfs(\Omega(D)),
        \end{align*}
        by Lemma~\ref{lemma-integral} again.
\end{proof}
\subsection{Preliminaries for the double induction}
In view of Proposition~\ref{prop-reduction}, it suffices to compute the volume $\volfs(\Omega(D))$ when $D$ has no zero eigenvalues. This will be accomplished
by a double induction on $p$ and $q$, the number of positive and negative eigenvalues. We introduce the following setup to efficiently handle 
the induction hypotheses as $p,q$ both vary.

Let $\mathbb{N}$ be the set of nonnegative integers, and recall that  $\mathbb{R}_+ = \{x \in \mathbb{R} : x>0\}.$
We introduce a collection $\mathscr{N}$ of double sequences of functions indexed by $\mathbb{N}^2$: $\mathscr{N}$ consists of double sequences $
    N = \{ N_{(p,q)} : (p,q) \in \mathbb{N}^2 \}$, where each $N_{(p,q)}$ is a continuous function
\[
    N_{(p,q)} : \mathbb{R}_{+}^{p} \times \mathbb{R}_{+}^{q} \longrightarrow \mathbb{R}_{+},
\] 
such that the function $(x,y)\mapsto N_{(p,q)}(x,y)$, where $x\in\mathbb{R}_{+}^{p}$ and $y\in\mathbb{R}_{+}^{q}$, is symmetric in the variables $x_1,\cdots,x_p$ and also symmetric in the variables $y_1,\cdots,y_q$, where $x=(x_1,\cdots,x_p)$ and $y=(y_1,\cdots,y_q)$, 
and satisfies the following properties:
\begin{itemize}[wide]
    \item \label{prop:N1} 
    \textbf{(Base Condition)} for all $x\in \mathbb{R}_{+}$,
    \begin{align}
    N_{(0,0)}(\cdot, \cdot)= 0 \quad \text{ and }
N_{(1,0)}(x, \cdot) &= 1.\quad \label{eq:base conditions}
\end{align}
  \item \textbf{(Duality)}  for all $(p,q) \in \mathbb{N}^2\setminus\{(0,0)\}$, and for all $x\in\mathbb{R}_{+}^p$ and $y\in\mathbb{R}_{+}^q$,
    \begin{equation}
        N_{(p,q)}(x,y)
        = \D_{(p,q)}(x,y) - N_{(q,p)}(y,x),
        \label{eq-npq duality relation}
    \end{equation}
    where $\D_{(p,q)}(x,y)$ is as in \eqref{eq-Ddef}.

    \item \label{prop:N3}
    \textbf{(Recursion)} for all $q\ge 1$, and for all $x\in\mathbb{R}_{+}^p$ and $y\in\mathbb{R}_{+}^q$ with  $y_1<y_2<\cdots<y_q$, 
    \begin{align}
        N_{(p,q)}(x,y)
        &= \left( \prod_{j=1}^{p} (x_j + y_q) \right)\cdot
           N_{(p,q-1)}(x, y[q])
           - \T_{(p,q)}(x,y)\cdot
             N_{(p,q-1)}(\alpha(x,y), \beta(y)),
        \label{eq-npq recursion}
    \end{align}
          
    with
    \begin{equation}
    \T_{(p,q)}(x,y)
        = \frac{1}{y_q^{pq - 2p - q + 1}}\cdot\prod\limits_{j=1}^{p} (y_q + x_j)^{q-1}
            \prod\limits_{k=1}^{q-1} (y_q - y_k)^{p-1},
        \label{eq-tpq def}
    \end{equation}
    \begin{align}
        \alpha(x,y) &= (\alpha_1, \dots, \alpha_p),
        &
        \alpha_j &= \frac{x_j}{y_q + x_j}, \quad 1 \le j \le p,\label{eq-defalpha}\\
         \beta(y) &= (\beta_1, \dots, \beta_{q-1}),
        &
        \beta_k &= \frac{y_k}{y_q - y_k}, \quad 1 \le k \le q-1.
        \label{eq-defbeta}
`    \end{align}
\label{prop:N4}
  
\end{itemize}
 
 With $\mathscr{N}$ as above,  will prove the following three propositions that will lead to the determination of $\volfs(\Omega(D))$ and the proof of the main theorem:
\begin{propa}\label{prop: uniquesness of N}
    The family $\mathscr{N}$ contains at most one element.
\end{propa}
\begin{propb}\label{Prop: N well defined and belongs to the set}
    For each $(p,q)\in \mathbb{N}^2\setminus\{(0,0)\}$, $x\in \mathbb{R}_{+}^{p}$ and $y\in\mathbb{R}_{+}^{q}$, let
     \begin{align}
        \n_{(p,q)}(x,y)=\frac{(p+q-1)!}{\pi^{p+q-1}}\cdot \D_{(p,q)}(x,y)\cdot \volfs(\Omega(\diag(x,-y)),\label{eq-npq def}
    \end{align}
    where $\diag(x,-y)=\diag(x_1,\cdots,x_p, -y_1,\cdots,-y_q)$ is a diagonal matrix of size $(p+q).$ If we set $ \n_{(0,0)}(\cdot, \cdot) = 0$,
 then the double sequence $\n=\{\n_{(p,q)}:(p,q)\in\mathbb{N}^2\}$  belongs to $\mathscr{N}$.
\end{propb}
\begin{propc}\label{prop: S is in N}
The double sequence  $\s=\{\s_{(p,q)}:(p,q)\in\mathbb{N}^2\}$ of \eqref{eq-Sdef} belongs to $\mathscr{N}$.
\end{propc}
\subsection{Proof of the main theorem assuming Propositions~A, B and C}
Postponing the proof of the three propositions, we can now complete the proof of the main theorem.  Let $A$ be as in the statement of the main theorem, and let $\Delta$ be its diagonal form as in \eqref{eq-Delta}. Combining \eqref{eq-diagonal} and \eqref{eq-DeltatoD} we see that
\begin{align*}
    \volfs(\Omega(A))&=\frac{\pi^r}{n(n-1)\dots (n-r+1)} \volfs(\Omega(\diag(\mu, -\nu)))\\
    &= \frac{\pi^r}{n(n-1)\dots (p+q)}\cdot \frac{\pi^{p+q-1}}{(p+q-1)!} \frac{\n_{(p,q)}(\mu,\nu)}{\D_{(p,q)}(\mu,\nu)} \quad\text{using definition~\ref{eq-npq def} of $\n$}\\
    &= \frac{\pi^n}{n!}\cdot\frac{\n_{(p,q)}(\mu,\nu)}{\D_{(p,q)}(\mu,\nu)}, \quad \text{ since $p+q+r=n+1$}.
\end{align*}

By Proposition B we get $\n\in \mathscr{N}$ and by Proposition C, we have $\s\in \mathscr{N}$. But Proposition A says that $\mathscr{N}$ contains at most one element, so $\s=\n$. Therefore, for each $(p,q)\in\mathbb{N}^2$, 
\[\volfs(\Omega(A))=\frac{\pi^n}{n!}\cdot\frac{\s_{(p,q)}(\mu,\nu)}{\D_{(p,q)}(\mu,\nu)},\] which is  \eqref{eq-explicit}, on noting from \eqref{eq-cpnvol} 
that $\frac{\pi^n}{n!}=\volfs(\cx\p^n).$
\section{The computations}\label{sec-computations}
\subsection{Proof of Proposition~A} Let $N,S\in\mathscr{N}$. It will suffice to show that $N_{(p,q)}(x,y)=S_{(p,q)}(x,y)$ for all $(p,q)\in \mathbb{N}^2, x\in \rl_+^p$
 and whenever
$y\in \rl_+^q$ is such that the components $y_1,\dots, y_q$ are distinct. Notice then we can use the recursion property \eqref{eq-npq recursion} since
the functions $N_{(p,q)}$ and $S_{(p,q)}$ are symmetric in $(x_1,\dots, x_p)$ and $(y_1,\dots, y_q)$. The general case when $y$ can have repeated entries can then be handled by the hypothesis of continuity of $N$ and $S$.

Using  \eqref{eq:base conditions} we have $N_{(0,0)}(\cdot,\cdot)=S_{(0,0)}(\cdot,\cdot)=0$. We now show by induction that 
    $N_{(0,q)}=S_{(0,q)}$, for all $q\ge 1$. We have for $x>0$
    \begin{align*}
        N_{(0,1)}(\cdot, x)&= \D_{(1,0)}(x,\cdot) - N_{(1,0)}(x,\cdot) \quad \text{ by duality}\\
        &= \D_{(1,0)}(x,\cdot) - S_{(1,0)}(x,\cdot) \quad \text{by the base condition  $N_{(1,0)}(x, \cdot)=S_{(1,0)}(x, \cdot) = 1$} \\
        &= S_{(0,1)}(\cdot, x) \quad \text{ by duality again.}
    \end{align*}
 This serves as the basis of the induction. Now   assume that $N_{(0,q)}=S_{(0,q)}$ for some $q$. Let $y'=(y_1,\cdots,y_{q+1})$. Using the recursion relation from equation \eqref{eq-npq recursion}, and the convention that empty products are 1, we obtain
\begin{align*}
    N_{(0,q+1)}(\cdot,y')&=1\cdot N_{(0,q)}(\cdot,y'[q+1])-\T_{(0,q+1)}(\cdot,y')\cdot N_{(0,q)}(\cdot,\beta(y'))\\
    &=1\cdot S_{(0,q)}(\cdot,y'[q+1])-\T_{(0,q+1)}(\cdot,y')\cdot S_{(0,q)}(\cdot,\beta(y')) \quad\text{(Induction hypothesis)}\\
    &=S_{(0,q+1)}(\cdot,y').
\end{align*}
Therefore, $N_{(0,q)}=S_{(0,q)}$ for all $q\ge 1$. Now it follows that $N_{(p,0)}=S_{(p,0)}$, for all $p\ge 1$, since  using the duality property 
\[    N_{(p,0)}(x,\cdot)=\D_{(p,0)}(x,\cdot)-N_{(0,p)}(\cdot,x)=\D_{(p,0)}(x,\cdot)-S_{(0,p)}(\cdot,x)=S_{(p,0)}(x,\cdot).
\]
This leaves the case $p\ge 1$ and $q\ge 1$, which is again done by induction.  Assume that for some $p\ge1, q\ge 1$,
$N_{(p,q)} = S_{(p,q)}.$ For fixed $p\ge 1$, by the recursion \eqref{eq-npq recursion}, for  $y'=(y_1,\cdots,y_{q+1})$
\[
\begin{aligned}
N_{(p,q+1)}(x,y')
&= \left(\prod_{j=1}^p (x_j + y_{q+1})\right)\cdot N_{(p,q)}(x, y'[q+1])
   - \T_{(p,q+1)}(x,y')\cdot N_{(p,q)}(\alpha(x,y'), \beta(y'))\\
&= \left(\prod_{j=1}^p (x_j + y_{q+1})\right)\cdot S_{(p,q)}(x, y'[q+1])
   - \T_{(p,q+1)}(x,y')\cdot S_{(p,q)}(\alpha(x,y'), \beta(y'))\\&\phantom{=====}\text{(by induction hypothesis)}\\
   &=S_{(p,q+1)}(x,y').
\end{aligned}
\]
Since $p \ge 1$ was arbitrary, we have $N_{(p,q)} = S_{(p,q)}$ for every $p \ge 1$ and every $q \ge 1$. Together with the previously established equalities for $p = 0$, this shows that $N_{(p,q)} = S_{(p,q)}$ for all $(p,q) \in \mathbb{N}^2$. Thus, $N = S$, and $\mathscr{N}$ contains at most one element.

\subsection{Proof of Proposition B} We collect two simple computations with the function $\D_{(p,q)}$ which will be needed below as well as in the proof of Proposition~C. Recalling the notation $x[q]$ 
from \eqref{eq-y[k]}, note that
 \begin{align}
        \D_{(p,q)}(x,y)&=\prod\limits_{k=1}^q\prod\limits_{j=1}^p(x_j+y_k)=\left(\prod\limits_{k=1}^{q-1}\prod\limits_{j=1}^p(x_j+y_k)\right)\cdot \prod\limits_{j=1}^p (x_j+y_q)\nonumber \\& =\D_{(p,q-1)}(x,y[q])\cdot \prod\limits_{j=1}^p (x_j+y_q) \label{eq-dpq first identity},
    \end{align}
and also recalling the meaning of $\alpha(x,y)$ and $\beta(y)$ from \eqref{eq-defalpha} and \eqref{eq-defbeta} above, we have
\begin{align}
    \D_{(p,q-1)}(\alpha(x,y),\beta(y))
&=\D_{(p,q-1)}\left(\left(\frac{x_j}{x_j+y_q}\right)_{j=1}^{p},\left(\frac{y_k}{y_q-y_k}\right)_{k=1}^{q-1}\right)\nonumber\\& = \prod_{k=1}^{q-1}\prod_{j=1}^{p}\left(\frac{x_j}{x_j+y_q} + \frac{y_k}{y_q-y_k}\right)\nonumber \\
& = \frac{y_q^{pq-p} \prod\limits_{k=1}^{q-1}\prod\limits_{j=1}^{p}(x_j+y_k)}{\prod\limits_{j=1}^p(x_j+y_q)^{q-1} \prod\limits_{k=1}^{q-1}(y_q-y_k)^p}\nonumber\\&=\frac{y_q^{pq-p}}{\prod\limits_{j=1}^p(x_j+y_q)^{q-1} \prod\limits_{k=1}^{q-1}(y_q-y_k)^p} \cdot \D_{(p,q-1)}(x,y[q])         \label{eq-dpqalphabeta identity}\\
&=\frac{y_q^{p+q-1}}{\prod\limits_{k=1}^{q-1}(y_q-y_k)}\cdot\frac{y_q^{pq-2p-q+1}}{\prod\limits_{j=1}^p(x_j+y_q)^{q-1} \prod\limits_{k=1}^{q-1}(y_q-y_k)^{p-1}}\cdot \D_{(p,q-1)}(x,y[q])\nonumber\\
&=\frac{y_q^{p+q-1}}{\prod\limits_{k=1}^{q-1}(y_q-y_k)}\cdot \frac{\D_{(p,q-1)}(x,y[q])}{\T_{(p,q)}(x,y)}\quad\text{(with $\T_{(p,q)}(x,y)$  as in \eqref{eq-tpq def})}\nonumber\\
&=\frac{y_q^{p+q-1}}{\prod\limits_{k=1}^{q-1}(y_q-y_k)\cdot \prod\limits_{j=1}^p(x_j+y_q)}\cdot \frac{\D_{(p,q)}(x,y)}{\T_{(p,q)}(x,y)} \label{eq-dpq 2nd identity}
\end{align}
where we have used  \eqref{eq-dpq first identity} to obtain \eqref{eq-dpqalphabeta identity} and \eqref{eq-dpq 2nd identity} from the previous steps.

Now we will verify that the double sequence $\n$ defined by \eqref{eq-npq def} belongs to the class $\mathscr{N}.$ Notice that it is clear from a routine application of the 
dominated convergence theorem that $\n_{(p,q)}$ is a continuous function of $p+q$ positive real variables. Further, since the Fubini-Study volume is invariant under projectivized unitary transformations, it is clear that the factor $\volfs(\Omega(\diag(x,-y)))$ is invariant under permutations of the variables $(x_1,\dots, x_p)$ and also under permutations of the variables $(y_1,\dots, y_q)$, i.e., it is symmetric in $x$ and $y$. Since $\D_{(p,q)}(x,y)$ also obviously has the same symmetry property, it follows that $\n_{(p,q)}(x,y)$ is also symmetric in $(x_1,\dots, x_p)$ and $(y_1,\dots, y_q)$. Let us now show that the double sequence $\n$ does satisfy the three defining conditions of the class $\mathscr{N}$.

\subsubsection{Verification of the base condition} The condition $\n_{(0,0)}(\cdot,\cdot)=0$ holds by definition. To compute $\n_{(1,0)}(x,\cdot)$ notice that in formula
\eqref{eq-npq def} defining this quantity, the set $\Omega(\diag(x,\cdot))$ is $\mathbb{CP}^0$ which has 
Fubini-Study volume 1 by definition, and $\D_{(1,0)}(x,\cdot)=1$ as an empty product, so 
    \begin{align*}
	\n_{(1,0)}(x,\cdot)=\frac{(1+0-1)!}{\pi^{1+0-1}}\cdot \D_{(1,0)}(x,\cdot)\cdot \volfs(\Omega(\diag(x,\cdot)))=1\cdot 1\cdot 1=1,
\end{align*}
 completing the proof.  
\subsubsection{Verification of Duality}
 If $p\geq 1, q=0$, then $\diag(x,\cdot)$ is positive definite and 
$\diag(\cdot,-x)$ is negative definite, so that $\Omega(\diag(x,\cdot))=\cx\p^{p-1}$ whereas
$\Omega(\diag(\cdot,-x))=\emptyset$. Therefore
\begin{align*}
 \n_{(p,0)}(x,\cdot)+ \n_{(0,p)}(\cdot,x)= \frac{(p-1)!}{\pi^{p-1}}\D_{(p,0)}(x,\cdot) \cdot   \volfs(\mathbb{CP}^{p-1})+0= \D_{(p,0)}(x,\cdot),
\end{align*}
using \eqref{eq-cpnvol}. The case $p=0,q\geq 1$ now follows since $\D_{(0,q)}(\cdot, y)= \D_{(p,0)}(x,\cdot)=1$ as empty products. 

Now assume $p,q\geq 1$ and let $x\in \rl_+^p, y\in \rl_+^q$. Let $D=\diag(x,-y)$, so that  we have a disjoint union representation:
\[  \mathbb{CP}^{p+q-1}=\Omega(D)\sqcup \Omega(-D)\sqcup H,\]
where $H$ is the hyperquadric surface $\{[z]\in\mathbb{CP}^{p+q-1}:\left(Dz,z\right)=0\}$. Since $p,q\geq 1$, the set $H$ has measure zero, and computing the volumes of the other sets we see that
\[ \volfs(\Omega(D))=\volfs{ \mathbb{CP}^{p+q-1}} - \volfs(\Omega(-D))=\frac{\pi^{p+q-1}}{(p+q-1)!}-\volfs(\Omega(-D)). \]
Therefore, from the definition \eqref{eq-npq def} we have
\begin{align*}
	\n_{(p,q)}(x,y)& =\frac{(p+q-1)!}{\pi^{p+q-1}}\cdot \D_{(p,q)}(x,y)\cdot \volfs(\Omega(\diag(x,-y))\\
	&=\frac{(p+q-1)!}{\pi^{p+q-1}}\cdot \D_{(p,q)}(x,y)\cdot \left(  \frac{\pi^{p+q-1}}{(p+q-1)!}-\volfs(\Omega(-D))\right)\\
	&= \D_{(p,q)}(x,y)- \frac{(q+p-1)!}{\pi^{q+p-1}}\D_{(q,p)}(y,x) \volfs(\Omega(\diag(y,-x)))\\
	&= \D_{(p,q)}(x,y) -\n_{(q,p)}(y,x),
	\end{align*}
where in the penultimate expression, we have used the fact that $\D_{(q,p)}(y,x)=\D_{(p,q)}(x,y)$.
This establishes  \eqref{eq-npq duality relation}.

    \subsubsection{Verification of the recursion relation} First, let $p\geq 1$. Let $a,b$ be defined in terms of $x,y$ as in \eqref{eq-deftuple"a"} and \eqref{eq-deftuple"b"}, and let  $E_{p-1,q}(a,b)$ be as in \eqref{eq-defsetE}. Using the notation   $t'=(t_1,\cdots,t_{p+q-2})\in \mathbb R_+^{p+q-2}$, we can rewrite the definition as:
\begin{align}
&E_{p-1,q}(a,b)\nonumber\\
&=\left\{(t',t_{p+q-1})\in \mathbb R_+^{p+q-2}\times \rl_+:0<t_{p+q-1}<\frac{1}{b_q}\left(1+\sum_{j=1}^{p-1}a_jt_j-\sum_{k=1}^{q-1}b_kt_{p+k-1}\right)\right\}\nonumber\\
&=\left\{(t',t_{p+q-1})\in \mathbb R_+^{p+q-2}\times \rl_+: t_{p+q-1}<\frac{1}{b_q}\left(1+\sum_{j=1}^{p-1}a_jt_j-\sum_{k=1}^{q-1}b_kt_{p+k-1}\right)\text{ and } \sum_{j=1}^{p-1}a_jt_j-\sum_{k=1}^{q-1}b_kt_{p+k-1}>-1\right\}\nonumber\\
&=\left\{(t',t_{p+q-1})\in \mathbb R_+^{p+q-2}\times \rl_+:t'\in E_{p-1,q-1}(a,b[q]), \, t_{p+q-1}<A(t')\right\},\label{eq-Epq2}
\end{align}
with  $A(t')=\dfrac{1}{b_q}\left(1+\sum\limits_{j=1}^{p-1}a_jt_j-\sum\limits_{k=1}^{q-1}b_kt_{p+k-1}\right),$
and 
    $b[q]=(b_1,\cdots,b_{q-1})$ is as in \eqref{eq-y[k]}.
    Now, from  Lemma~\ref{lemma-integral} (see \eqref{eq-V(p,q)asintegral}), we have, setting $\displaystyle{s(t')=\sum_{j=1}^{p+q-2} t_j}$:
    \begin{align*}
    	&\volfs(\Omega(\diag(x,-y)))\\&= \pi^{p+q-1} \int\limits_{E_{p-1,q}(a,b)} \frac{dV(t)}{\left(1 + s(t')+ t_{p+q-1}  \right)^{p+q}}.\\
    	&=  \pi^{p+q-1} \int\limits_{t'\in E_{p-1,q-1}(a,b[q])}  \left(\,\,\int\limits_{t_{p+q-1}=0}^{A(t')}\frac{dt_{p+q-1}}{\left(1 + s(t')+ t_{p+q-1}  \right)^{p+q}}\right)dV(t'),\text{ using  \eqref{eq-Epq2}}\\
    	& =\frac{\pi^{p+q-1}}{p+q-1} \cdot\int\limits_{t'\in E_{p-1,q-1}(a,b[q])}\left(\frac{1}{\left(1+s(t')\right)^{p+q-1}} -\frac{1}{\left(\left(1+\frac{1}{b_q}\right)+\sum\limits_{j=1}^{p-1}\left(1+\frac{a_j}{b_q}\right)t_j+\sum\limits_{k=1}^{q-1}\left(1-\frac{b_k}{b_q}\right)t_{p+k-1}\right)^{p+q-1}}\right)dV(t'),
    \end{align*}
    by an elementary computation of the inner integral in the variable $t_{p+q-1}$. 
Therefore, from the definition \eqref{eq-npq def} of $\n_{(p,q)}$:
\begin{align}
	\n_{(p,q)}(x,y)&=\frac{(p+q-1)!}{\pi^{p+q-1}}\cdot \D_{(p,q)}(x,y)\cdot\volfs(\Omega(\diag(x,-y)))=\mathscr{J}-\mathscr{K},\label{eq-integralstepforNp,q-1}
\end{align}

%
where \begin{align}
    \mathscr{J}&={(p+q-2)!}\cdot \D_{(p,q)}(x,y)\cdot \int\limits_{ E_{p-1,q-1}(a,b[q])}\frac{dV(t')}{\left(1+\sum_{j=1}^{p+q-2} t_j\right)^{p+q-1}}\nonumber\\
   & =\frac{(p+q-2)!}{\pi^{p+q-2}}\cdot \prod\limits_{j=1}^p (x_j+y_q)\cdot \D_{(p,q-1)}(x,y[q])\cdot \volfs(\Omega(\diag(x,-y[q])))\nonumber\\ &\text{ (using \eqref{eq-dpq first identity} and Lemma~\ref{lemma-integral})}\nonumber\\
   &=\left(\prod\limits_{j=1}^p (x_j+y_q)\right)\cdot \n_{(p,q-1)}(x,y[q]),\label{eq-first term of recursion}
\end{align}
and 
\begin{align}
\mathscr{K}
&={(p+q-2)!}\cdot \D_{(p,q)}(x,y)\cdot  I,\label{eq-Kintegral}
   \end{align}
   with \begin{align}
       I&=\int\limits_{ E_{p-1,q-1}(a,b[q])}\frac{dV(t')}{\left(\left(1+\frac{1}{b_q}\right)+\sum\limits_{j=1}^{p-1}\left(1+\frac{a_j}{b_q}\right)t_j+\sum\limits_{k=1}^{q-1}\left(1-\frac{b_k}{b_q}\right)t_{p+k-1}\right)^{p+q-1}}.\nonumber\\
&=\left(\frac{b_q}{b_q+1}\right)^{p+q-1}\int\limits_{ E_{p-1,q-1}(a,b[q])}\frac{dV(t')}{\left(1+\sum\limits_{j=1}^{p-1}\frac{b_q+a_j}{b_q+1}t_j+\sum\limits_{k=1}^{q-1}\frac{b_q-b_k}{b_q+1}t_{p+k-1}\right)^{p+q-1}}.\label{eq-Iintegral}
\end{align}
To compute the integral $I$,  make a linear change of variables in $\rl^{p+q-2}$ given by $u=Mt$ where 
 $M$ is the $(p+q-2)\times (p+q-2)$ diagonal matrix with diagonal entries
\[M_{\ell\ell}=
\begin{cases}
	\dfrac{b_q+a_\ell}{b_q+1} & \text{for } 1 \le \ell \le p-1 \\
	\dfrac{b_q-b_{\ell-p+1}}{b_q+1} & \text{for } p \le \ell \le p+q-2.
\end{cases}\]
Then we have
\begin{equation}
    I=\eqref{eq-Iintegral}= \left(\frac{b_q}{b_q+1}\right)^{p+q-1}\cdot\frac{1}{\abs{\det M}}\cdot\int\limits_{ M(E_{p-1,q-1}(a,b[q]))}\frac{dV(u)}{\left(1+s(u)\right)^{p+q-1}}.\label{eq-Iintegral2}
\end{equation}
Now, in the recursion relation \eqref{eq-npq recursion} that we must show holds for $\n_{(p,q)}$, since it is assumed that $x_j$ and $y_k$ are positive, we have $a_j=x_{j+1}/x_1 > 0,$ for $1\le j\le p-1$, $b_k=y_k/x_1 > 0$, for $1\le k\le q$ (see the definition of the tuples $a$ and $b$ in \eqref{eq-deftuple"a"} and \eqref{eq-deftuple"b"}). Further, since in 
\eqref{eq-npq recursion} it is assumed that $y_1<y_2<\cdots<y_q$, it follows from the definition of $b$ that $b_q-b_k = (y_q-y_k)/x_1 > 0$ for $1 \le k \le q-1$. Consequently each
diagonal entry $M_{\ell\ell}$ of the matrix $M$ of the change of variables is strictly positive, and we have
\begin{equation}
    \label{eq-detM}
    \abs{\det M}= \prod\limits_{\ell=1}^{p+q-2} M_{\ell\ell}
= \frac{\prod\limits_{j=1}^{p-1}(b_q+a_j)\prod\limits_{k=1}^{q-1}(b_q-b_k)}{(b_q+1)^{p+q-2}}.
\end{equation}
To compute \eqref{eq-Iintegral2}, we note that 
\begin{equation}
    M(E_{p-1,q-1}(a,b[q]))=E_{p-1,q-1}(d,e)\label{eq-ME}
\end{equation}
where, with $\alpha(x,y)$ and $\beta(y)$ as in are as in  \eqref{eq-defalpha} and \eqref{eq-defbeta},
\begin{align}
d= \frac{\alpha(x,y)[1]}{\alpha_1}, \quad \text{ i.e. }  d_j= \frac{\alpha_{j+1}}{\alpha_1},\quad \text{for } 1\le j\le p-1, \label{eq-defd_j}
\end{align}
and
\begin{align}
e= \frac{\beta(y)}{\alpha_1}, \quad \text{ i.e. }  e_k= \frac{\beta_k}{\alpha_1},\quad \text{for } 1\le k\le q-1\label{eq-defe_k}.
\end{align}

To see \eqref{eq-ME}, we first note that using $\alpha_j=\dfrac{x_j}{y_q+x_j}$ for $1\le j\le p$ and $\beta_k=\dfrac{y_k}{y_q-y_k}$ for $1\le k \le q-1$, we obtain $d_j=\dfrac{a_j(b_q+1)}{b_q+a_j}$ and $e_k=\dfrac{b_k(b_q+1)}{b_q-b_k}$ (recall from \eqref{eq-deftuple"a"} and \eqref{eq-deftuple"b"} that $a=\dfrac{x[1]}{x_1}, b=\dfrac{y}{x_1}$.)
Therefore, if we set $u_j=M_{jj}t_j$
 for $1\le j\le p+q-2$,
\begin{align}
 \sum\limits_{j=1}^{p-1}d_ju_j-\sum\limits_{k=1}^{q-1} e_ku_{p+k-1}\nonumber
 &=\sum\limits_{j=1}^{p-1}d_j\frac{b_q+a_j}{b_q+1}t_j-\sum\limits_{k=1}^{q-1} e_k\frac{b_q-b_k}{b_q+1}t_{p+k-1}\nonumber\\
 &=\sum\limits_{j=1}^{p-1}\left(\frac{a_j(b_q+1)}{b_q+a_j}\right)\left(\frac{b_q+a_j}{b_q+1}\right)t_j-\sum\limits_{k=1}^{q-1} \left(\frac{b_k(b_q+1)}{b_q-b_k}\right)\left(\frac{b_q-b_k}{b_q+1}\right)t_{p+k-1}\nonumber\\
 &=\sum\limits_{j=1}^{p-1}a_jt_j-\sum\limits_{k=1}^{q-1}b_kt_{p+k-1},\label{eq-Mcomputation}
\end{align}
Recalling from \eqref{eq-defsetE} the definition of the sets  $E_{p-1,q-1}(a,b[q]))$ and $E_{p-1,q-1}(d,e)$, the equality \eqref{eq-ME} now follows.  Therefore, using \eqref{eq-detM} and \eqref{eq-ME}, the  integral $I$ in \eqref{eq-Iintegral} transformed as follows:
\begin{align}
I=\eqref{eq-Iintegral2}
&= \frac{b_q^{p+q-1}}{(b_q+1)^{p+q-1}}\cdot\frac{(b_q+1)^{p+q-2}}{\prod\limits_{j=1}^{p-1}(b_q+a_j)\prod\limits_{k=1}^{q-1}(b_q-b_k)}\cdot\int\limits_{E_{p-1,q-1}(d,e)}\frac{dV(u)}{\left(1+s(u)\right)^{p+q-1}}\nonumber\\
&=\frac{y_q^{p+q-1}}{\prod\limits_{j=1}^{p}(y_q+x_j)\cdot\prod\limits_{k=1}^{q-1}(y_q-y_k)}\cdot\int\limits_{E_{p-1,q-1}(d,e)}\frac{dV(u)}{\left(1+s(u)\right)^{p+q-1}} \label{eq-transformedintegral}
\end{align}
where \eqref{eq-transformedintegral} is obtained from the previous step by substituting $a_j=\dfrac{x_{j+1}}{x_1}$ from \eqref{eq-deftuple"a"} and $b_k=\dfrac{y_k}{x_1}$ from \eqref{eq-deftuple"b"}. Now, from \eqref{eq-dpq 2nd identity} it follows that
\begin{align}
    \D_{(p,q)}(x,y)&= \frac{\prod\limits_{j=1}^p(x_j+y_q)\cdot\prod\limits_{k=1}^{q-1}(y_q-y_k)}{y_q^{p+q-1}}\cdot  \T_{(p,q)}(x,y)\cdot \D_{(p,q-1)}(\alpha(x,y),\beta(y)).\label{eq-dpqxy in terms of dpqalphabeta}
\end{align}
Therefore, substituting the results from equations \eqref{eq-dpqxy in terms of dpqalphabeta} and \eqref{eq-transformedintegral} into $\mathscr{K}$, we obtain
\begin{align}
 \mathscr{K}
 &={(p+q-2)!}\cdot\frac{\prod\limits_{j=1}^p(x_j+y_q)\cdot\prod\limits_{k=1}^{q-1}(y_q-y_k)}{y_q^{p+q-1}}\cdot  \T_{(p,q)}(x,y)\cdot \D_{(p,q-1)}(\alpha(x,y),\beta(y))\nonumber\\
    &\cdot \frac{y_q^{p+q-1}}{\prod\limits_{j=1}^{p}(y_q+x_j)\cdot\prod\limits_{k=1}^{q-1}(y_q-y_k)}\cdot \int\limits_{E_{p-1,q-1}(d,e)}\frac{dV(u)}{\left(1+s(u)\right)^{p+q-1}}\nonumber\\
    &=\T_{(p,q)}(x,y)\cdot \frac{(p+q-2)!}{\pi^{p+q-2}}\cdot \D_{(p,q-1)}(\alpha(x,y),\beta(y))\cdot \pi^{p+q-2}\cdot \int\limits_{E_{p-1,q-1}\left(\frac{\alpha(x,y)[1]}{\alpha_1}, \frac{\beta(y)}{\alpha_1} \right)}\frac{dV(u)}{\left(1+s(u)\right)^{p+q-1}}\nonumber\\
    &=\T_{(p,q)}(x,y)\cdot \frac{(p+q-2)!}{\pi^{p+q-2}}\cdot \D_{(p,q-1)}(\alpha(x,y),\beta(y))\cdot \volfs(\Omega(\diag(\alpha(x,y),-\beta(y)))) \label{eq-usingintegral}\\
    &=\T_{(p,q)}(x,y)\cdot \n_{(p,q-1)}(\alpha(x,y),\beta(y)),\label{eq-second term of recursion}
\end{align}
where, with  $\T_{(p,q)}(x,y)$ is as in \eqref{eq-tpq def}, in \eqref{eq-usingintegral} we have used the integral representation \eqref{eq-V(p,q)asintegral} from Lemma~\ref{lemma-integral}, and in \eqref{eq-second term of recursion} we have used the definition~\ref{eq-npq def} of $\n_{(p,q)}.$ Plugging in \eqref{eq-first term of recursion} and \eqref{eq-second term of recursion} into \eqref{eq-integralstepforNp,q-1}, the recursion formula \eqref{eq-npq recursion} follows when $p\geq 1$.

 To complete the proof, we verify the recursion relation for $p=0, q\geq 1$. Then for $y\in \rl_+^q$, the quadric domain $\Omega(\diag(\cdot, -y))$ is the empty subset of $\cx^{q-1},$
 and consequently
 \[ \n_{(0,q)}(\cdot,y)=\frac{(q+0-1)!}{\pi^{q+0-1}}\cdot \D_{(0,q)}(\cdot,y)\cdot \volfs(\Omega(\diag(\cdot,-y)))=0, \]
noting that $\D_{(0,q)}(\cdot,y)=1$ as an empty product. It follows that in \eqref{eq-npq recursion}, the left hand side is zero, and each of the two terms on the right hand side 
is also zero. This completes the proof of the recursion relation, and therefore of Proposition B.

\subsection{Proof of Proposition C}
    We verify the four defining properties of the class $\mathscr{N}$  for the functions $\s_{(p,q)}(x,y)= \sum_{\lambda\in \cc(p,q)}s_\lambda(x) s_{\lambda^*_{p,q}}(y).$ The continuity of the $\s_{(p,q)}$ is trivial, since the $\s_{(p,q)}$ are polynomials, and so is the symmetry, since the Schur polynomials being symmetric, each summand in the defining sum  $\sum_{\xi\in\mathfrak{C}(p,q)}s_\xi(x) s_{\xi^\ast_{p,q}}(y)$ of $\s_{(p,q)}$ is symmetric in the variables $x=(x_1,\dots, x_p)$ and $y=(y_1,\dots, y_q).$

    \subsubsection{ Verification of the base conditions of \eqref{eq:base conditions}}  Notice that $\cc(0,0)$ is the collection of partitions with no parts in which the largest part is 0, so that $\cc(0,0)$ is the empty set. Therefore $\s_{(0,0)}=0,$ as an empty sum. 

Now $\cc(1,0)$ consists of partitions having at most one part in which the largest part is 0, so $\cc(1,0)$ consists only of the partition
$(0)$, the empty partition with no nonzero parts.    By the definition \eqref{eq-starpqdef}, $(0)^*_{1,0}=(0).$ The Schur polynomial for the empty partition is $s_{(0)}=1$, so 
\[
\s_{(1,0)}(x,\cdot) = \sum_{\xi \in \mathfrak{C}(1,0)} s_\xi(x) s_{\xi_{1,0}^*}(\cdot) = s_0(x) \cdot s_0(\cdot) = 1 \cdot 1 = 1,
\]
verifying the base conditions of $\mathscr{N}$.

\subsubsection{Verification of the duality relation, \eqref{eq-npq duality relation}} We need the following combinatorial lemma:
\begin{lem}
    \label{lem-combinatorial2}
    The mapping $\lambda\mapsto \lambda^{\ast}_{q,p}$ is a bijection from 
    $\mathfrak{C}(q,p)$ onto  $\mathfrak{B}(p,q) \setminus \mathfrak{C}(p,q)$, 
    whose inverse is the map $\mu \mapsto \mu^*_{p,q}.$
\end{lem}

\begin{proof} Set $f(\lambda)=\lambda^{\ast}_{q,p}$.
Let $\lambda \in \mathfrak{C}(q,p)$  so that $\lambda_1=p$, which implies $\lambda'_p = \#\{j: \lambda_j \ge p\} \ge 1$.
The largest part of $\lambda_{q,p}^*$ is by definition $(\lambda_{q,p}^*)_1 = q - \lambda'_p\leq q-1$. It follows that 
$\lambda_{q,p}^* \not \in \mathfrak{C}(p,q)$ and therefore $f$ maps into $\mathfrak{B}(p,q) \setminus \mathfrak{C}(p,q)$.

To show injectivity, assume $f(\lambda) = f(\nu)$ for $\lambda, \nu \in \mathfrak{C}(q,p)$. Then by the definition \eqref{eq-starpqdef}:
\[(q-\lambda_p', \dots, q-\lambda_1') = (q-\nu_p', \dots, q-\nu_1').\]
Equating the components shows $\lambda'_k = \nu'_k$ for $k=1, \dots, p$.
Since $\lambda', \nu'$ are partitions with at most $p$ non-zero parts, $\lambda' = \nu'$.
As $\lambda''=\lambda$, we conclude $\lambda = \nu$.
Therefore, $f$ is injective.

To show surjectivity, let $\mu \in \mathfrak{B}(p,q) \setminus \mathfrak{C}(p,q)$,  and we want a $\lambda\in \mathfrak{C}(q,p)$ satisfying 
$f(\lambda)=\mu$, i.e., $\lambda'_k=q- \mu_{p-k+1}$ for $1\leq k \leq p$. Since the conjugation operation on partitions is involutive, we can take
$\lambda=\nu'$ where $\nu_k=q-\mu_{p-k+1}.$ Now since $\mu \not \in \mathfrak{C}(p,q)$, we have $\mu_1 \leq q-1$, so that $\nu_p=q-\mu_1\geq 1.$
Consequently $\lambda_1=\nu'_1= \#\{j: \nu_j \ge 1\}=p$, which shows that $\lambda\in \mathfrak{C}(q,p).$
Hence, the map $f$ is a bijection. 

To compute the inverse of $f$, let $\mu=f(\lambda)$, i.e. $\mu_j=q-\lambda_{p-j+1}$ for $1\leq j \leq p$. From this we have $\lambda_\ell'=q-\mu_{p-\ell+1}, 1 \leq \ell \leq p$. Since the conjugation map $\lambda\mapsto \lambda'$ is its own inverse, we have 
\begin{align*}
    \lambda_k & =\# \{1 \leq \ell\leq p:\lambda_\ell'\geq k  \}=\# \{1 \leq \ell\leq p:q-\mu_{p-\ell+1}\geq k\}\\
    &= \# \{1 \leq \ell\leq p: \mu_\ell \leq q-k\}\\
    &= p- \# \{1 \leq \ell\leq p:\mu_{\ell}\geq q-k+1\}\\
    &= p- \mu_{q-k+1}',
\end{align*}
i.e. $\lambda =\mu_{p,q}^*.$
\end{proof}
Using Lemma~\ref{lem-combinatorial2}, we have
\[\s_{(q,p)}(y,x)=  \sum\limits_{\eta\in\mathfrak{C}(q,p)}s_\eta(y) s_{\eta^\ast_{q,p}}(x) = \sum\limits_{\mu\in\mathfrak{B}(p,q)\setminus\mathfrak{C}(p,q)}s_{\xi}(x) s_{\xi^\ast_{p,q}}(y),\]
where in the last step we set $\xi=\eta^*_{q,p}$. Therefore
        \begin{align*}
          \s_{(p,q)}(x,y)+\s_{(q,p)}(y,x)
          &=\sum\limits_{\xi\in\mathfrak{C}(p,q)}s_\xi(x) s_{\xi^\ast_{p,q}}(y)+\sum\limits_{\xi\in\mathfrak{B}(p,q)\setminus\mathfrak{C}(p,q)}s_{\xi}(x) s_{\xi^\ast_{p,q}}(y)=\sum\limits_{\xi\in\mathfrak{B}(p,q)}s_{\xi}(x) s_{\xi^\ast_{p,q}}(y)\\
          &=\D_{(p,q)}(x,y),\quad \text{ by \eqref{eq-dualcauchy},}
\end{align*}
completing the verification of the duality condition. 
\subsubsection{Verification of the recursion relation, \eqref{eq-npq recursion}}  We  want to show  \eqref{eq-npq recursion} for $\s$, i.e.,  
        \begin{align}
            \s_{(p,q)}(x,y)&
            = \underbrace{\left(\prod\limits\limits_{j=1}^{p}(x_j+y_q)\right)\cdot \s_{(p,q-1)}(x,y[q])}_G
            - \underbrace{\T_{(p,q)}(x,y)\cdot \s_{(p,q-1)}(\alpha(x,y),\beta(y)).}_H \label{eq-Spq recursion}
    \end{align}
First assume $p\geq 1$. Using Proposition~\ref{prop-spqalternative} (see \eqref{eq-Spq in terms of sum}), we have
    \begin{align}
        G&=\sum\limits_{\ell=1}^p \frac{x_\ell^{p+q-2}}{\prod\limits_{\substack{k=1\\k\ne\ell}}^p(x_\ell-x_k)}\prod\limits\limits_{j=1}^{p}(x_j+y_q)\cdot \D_{(p-1,q-1)}(x[\ell],y[q])\nonumber\\
       &=\sum\limits_{\ell=1}^p \frac{x_\ell^{p+q-2}}{\prod\limits_{\substack{k=1\\k\ne\ell}}^p(x_\ell-x_k)}\cdot (x_\ell+y_q)\cdot\prod\limits\limits_{\substack{j=1 \\ j \neq l}}^{p}(x_j+y_q)\cdot \D_{(p-1,q-1)}(x[\ell],y[q])\nonumber\\
        &=\sum\limits_{\ell=1}^p \frac{x_\ell^{p+q-2}}{\prod\limits_{\substack{k=1\\k\ne\ell}}^p(x_\ell-x_k)}(x_\ell+y_q)\cdot \D_{(p-1,q)}(x[\ell],y) \quad\text{(Using \eqref{eq-dpq first identity})}\nonumber\\
        &=\sum\limits_{\ell=1}^p x_\ell^{p+q-2}\cdot \frac{(x_\ell+y_q)}{\prod\limits\limits_{\substack{k=1 \\ k \neq l}}^{p}(x_\ell-x_k)}\cdot \D_{(p-1,q)}(x[\ell],y).\label{eq-G Simplified}
    \end{align}

    Before proceeding  with the simplification of $H$, notice that from \eqref{eq-defalpha} we have
    \begin{align*}
        \alpha(x,y)[\ell]&= (\alpha_1, \dots, \alpha_{\ell-1}, \alpha_{\ell+1}, \dots, \alpha_p)\\
        &= \left(\frac{x_1}{y_q+x_1}, \dots, \frac{x_{\ell-1}}{y_q+x_{\ell-1}}, \frac{x_{\ell+1}}{y_q+x_{\ell+1}}, \dots, \frac{x_p}{y_q+x_p} \right)\\
        &= \alpha(x[\ell],y).
    \end{align*}
    Therefore, from \eqref{eq-dpqalphabeta identity}, for each $\ell$:
    \begin{align}
        \D_{(p-1,q-1)}(\alpha(x,y)[\ell],\beta(y))\nonumber&=\D_{(p-1,q-1)}(\alpha(x[\ell],y),\beta(y))\nonumber\\
        &=\frac{y_q^{pq-p-q+1} }{\prod\limits_{\substack{j=1\\j\ne\ell}}^p(x_j+y_q)^{q-1} \prod\limits_{k=1}^{q-1}(y_q-y_k)^{p-1}}\cdot \D_{(p-1,q-1)}(x[\ell],y[q]).\label{eq-D(p-1,q-1)(alpha[l],beta)}
    \end{align}
  Therefore
    \begin{align}
        H&=\T_{(p,q)}(x,y)\cdot\sum\limits_{\ell=1}^p \frac{\alpha_\ell^{p+q-2}}{\prod\limits_{\substack{k=1\\k\ne\ell}}^p(\alpha_\ell-\alpha_k)}\D_{(p-1,q-1)}(\alpha(x,y)[\ell],\beta(y))\nonumber\\
        &=\T_{(p,q)}(x,y)\sum\limits_{\ell=1}^p \frac{\alpha_\ell^{p+q-2}}{\prod\limits_{\substack{k=1\\k\ne\ell}}^p(\alpha_\ell-\alpha_k)}\cdot \frac{y_q^{pq-p-q+1} }{\prod\limits_{\substack{j=1\\j\ne\ell}}^p(x_j+y_q)^{q-1} \prod\limits_{k=1}^{q-1}(y_q-y_k)^{p-1}}\cdot \D_{(p-1,q-1)}(x[\ell],y[q])\nonumber\\&\phantom{==}\text{(Using \eqref{eq-D(p-1,q-1)(alpha[l],beta)})} \nonumber\\
        &=y_q^p\cdot \prod\limits\limits_{j=1}^{p}(x_j+y_q)^{q-1}\sum\limits_{\ell=1}^p \frac{\alpha_\ell^{p+q-2}}{\prod\limits_{\substack{k=1\\k\ne\ell}}^p(\alpha_\ell-\alpha_k)}\cdot \frac{1}{\prod\limits_{\substack{j=1\\j\ne\ell}}^p(x_j+y_q)^{q-1}}\cdot \D_{(p-1,q-1)}(x[\ell],y[q])\nonumber\\
        &=y_q^p\sum\limits_{\ell=1}^p \frac{\alpha_\ell^{p+q-2}}{\prod\limits_{\substack{k=1\\k\ne\ell}}^p(\alpha_\ell-\alpha_k)}\cdot \frac{\prod\limits\limits_{j=1}^{p}(x_j+y_q)^{q-1}}{\prod\limits_{\substack{j=1\\j\ne\ell}}^p(x_j+y_q)^{q-1}}\cdot \D_{(p-1,q-1)}(x[\ell],y[q])\nonumber\\
        &=y_q^p\sum\limits_{\ell=1}^p\left(\frac{x_\ell}{x_\ell+y_q}\right)^{p+q-2}\cdot\frac{(x_\ell+y_q)^{q-1}}{\prod\limits\limits_{\substack{k=1\\k\ne\ell}}^p(\alpha_\ell-\alpha_k)}\cdot \D_{(p-1,q-1)}(x[\ell],y[q])\nonumber\\
        &=y_q^p\sum\limits_{\ell=1}^p\frac{x_\ell^{p+q-2}}{(x_\ell+y_q)^{p+q-2}}\cdot (x_\ell+y_q)^{q-1}\frac{1}{\prod\limits\limits_{\substack{k=1\\k\ne\ell}}^p\left(\frac{x_\ell}{x_\ell+y_q}-\frac{x_k}{x_k+y_q}\right)}\cdot \D_{(p-1,q-1)}(x[\ell],y[q])\nonumber\\ &\phantom{==}\text{(Recall $\alpha_j=\frac{x_j}{x_j+y_q}$ from \eqref{eq-defalpha})}\nonumber\\
        &=y_q^p\sum\limits_{\ell=1}^p\frac{x_\ell^{p+q-2}}{(x_\ell+y_q)^{p-1}}\cdot\frac{\prod\limits\limits_{\substack{k=1\\k\ne\ell}}^p\left((x_\ell+y_q)(x_k+y_q)\right)}{y_q^{p-1}\prod\limits\limits_{\substack{k=1\\k\ne\ell}}^p(x_\ell-x_k)}\cdot \D_{(p-1,q-1)}(x[\ell],y[q])\nonumber\\
        &=y_q\sum\limits_{\ell=1}^p\frac{x_\ell^{p+q-2}}{(x_\ell+y_q)^{p-1}}\cdot\frac{(x_\ell+y_q)^{p-1}\prod\limits\limits_{\substack{k=1\\k\ne\ell}}^p(x_k+y_q)}{\prod\limits\limits_{\substack{k=1\\k\ne\ell}}^p(x_\ell-x_k)}\cdot \D_{(p-1,q-1)}(x[\ell],y[q])\nonumber\\
        &=y_q\sum\limits_{\ell=1}^px_\ell^{p+q-2}\cdot\frac{\prod\limits\limits_{\substack{k=1\\k\ne\ell}}^p(x_k+y_q)}{\prod\limits\limits_{\substack{k=1\\k\ne\ell}}^p(x_\ell-x_k)}\cdot \D_{(p-1,q-1)}(x[\ell],y[q]).\label{eq-Simplified H}
    \end{align}
  Therefore, we have $G-H=\eqref{eq-G Simplified}-\eqref{eq-Simplified H} =$
    \begin{align*}
        &\phantom{=}\sum\limits_{\ell=1}^p x_\ell^{p+q-2}\cdot \frac{(x_\ell+y_q)}{\prod\limits\limits_{\substack{k=1 \\ k \neq l}}^{p}(x_\ell-x_k)}\cdot \D_{(p-1,q)}(x[\ell],y)-y_q\sum\limits_{\ell=1}^p\frac{x_\ell^{p+q-2}\prod\limits\limits_{\substack{k=1\\k\ne\ell}}^p(x_k+y_q)}{\prod\limits\limits_{\substack{k=1\\k\ne\ell}}^p(x_\ell-x_k)}\cdot \D_{(p-1,q-1)}(x[\ell],y[q])\\
        &=\sum\limits_{\ell=1}^p\frac{x_\ell^{p+q-2}}{\prod\limits\limits_{\substack{k=1\\k\ne\ell}}^p(x_\ell-x_k)}\left[(x_\ell+y_q)\cdot \D_{(p-1,q)}(x[\ell],y)-y_q\cdot \prod\limits\limits_{\substack{k=1\\k\ne\ell}}^p(x_k+y_q)\cdot \D_{(p-1,q-1)}(x[\ell],y[q])\right]\\
        &=\sum\limits_{\ell=1}^p\frac{x_\ell^{p+q-2}}{\prod\limits\limits_{\substack{k=1\\k\ne\ell}}^p(x_\ell-x_k)}\left[(x_\ell+y_q)\cdot \D_{(p-1,q)}(x[\ell],y)-y_q\cdot \D_{(p-1,q)}(x[\ell],y)\right]\quad \text{Using  \eqref{eq-dpq first identity}}\\
        &=\sum\limits_{\ell=1}^p\frac{x_\ell^{p+q-2}}{\prod\limits\limits_{\substack{k=1\\k\ne\ell}}^p(x_\ell-x_k)}\cdot x_\ell\cdot \D_{(p-1,q)}(x[\ell],y)\\
        &=\sum\limits_{\ell=1}^p\frac{x_\ell^{p+q-1}}{\prod\limits\limits_{\substack{k=1\\k\ne\ell}}^p(x_\ell-x_k)}\cdot \D_{(p-1,q)}(x[\ell],y)= \s_{(p,q)}(x,y) \quad \text{ using Proposition~\ref{prop-spqalternative}}.
    \end{align*}
    We still need to verify the recursion relation in the case $p=0$ not covered in the computation above. In this case the sum defining $\s_{0,q}$ is empty, as the set $\cc(0,q)$ (partitions with no parts but with largest part equal to $q$) is empty, and therefore we have $\s_{(0,q)}(\cdot,y) = 0.$ On the other hand the right hand side of the recursion is
    \begin{align*}
        \left(\prod\limits\limits_{j=1}^{0}(x_j+y_q)\right)\cdot \s_{(0,q-1)}(\cdot,y[q])-\T_{(0,q)}(\cdot,y)\cdot \s_{(0,q)}(\cdot,\beta(y))
        =1\cdot0-\T_{(0,q)}(\cdot,y)\cdot0 
        =0,
    \end{align*}
    using the fact $\s_{(0,q)}(\cdot,y) = 0$ one more time.

\bibliographystyle{alpha}
\bibliography{references}

\end{document}